\documentclass[a4paper,12pt,twoside,notitlepage,reqno]{amsart}

\usepackage{amsmath,amssymb,amsbsy,amsthm}
\usepackage[hmargin=1.4in,vmargin=1.4in]{geometry}
\usepackage{url}

\usepackage{tikz}
\usetikzlibrary{matrix}
\usepackage[linktocpage]{hyperref}
\hypersetup{
    colorlinks=true,        
    linkcolor=red,          
    citecolor=red,         
    filecolor=magenta,      
    urlcolor=cyan           
}

\addtocontents{toc}{\protect\setcounter{tocdepth}{1}}

\usepackage{eucal}

\newcommand{\Q}{ \mathbb{Q} }
\newcommand{\Z}{ \mathbb{Z} }
\newcommand{\N}{ \mathbb{N} }
\newcommand{\F}{ \mathbb{F} }

\theoremstyle{plain}
	\newtheorem{thm}{Theorem}
	
		\numberwithin{thm}{section}
	\newtheorem{lemma}[thm]{Lemma}
	
	\newtheorem{cor}[thm]{Corollary}
	\newtheorem{conj}[thm]{Conjecture}

	\newtheorem*{thm*}{Theorem}
	\newtheorem*{lemma*}{Lemma}
	\newtheorem*{prop*}{Proposition}
	\newtheorem*{cor*}{Corollary}
	\newtheorem*{conj*}{Conjecture}

\theoremstyle{definition}
	
	\newtheorem*{example*}{Example}

\begin{document}


\title[Variant of the dynamical Mordell-Lang conjecture]{A fusion variant of the classical and dynamical Mordell-Lang conjectures in positive characteristic}

\author{Jason Bell}
\address{University of Waterloo \\
Department of Pure Mathematics \\
Waterloo, Ontario \\
N2L 3G1, Canada}
\email{jpbell@uwaterloo.ca}

\author{Dragos Ghioca}
\address{University of British Columbia\\
Department of Mathematics\\
Vancouver, BC\\
V6T 1Z2, Canada}
\email{dghioca@math.ubc.ca}

\thanks{The authors were partially supported by Discovery Grants from the National Science and Engineering Research Council of Canada.}

\begin{abstract}
We study an open question at the interplay between the classical and the dynamical Mordell-Lang conjectures in positive characteristic. 
Let $K$ be an algebraically closed field of positive characteristic, let $G$ be
a finitely generated subgroup of the multiplicative group of $K$, and let $X$ be a (irreducible) quasiprojective variety defined over $K$.  We consider $K$-valued sequences of the form
$a_n:=f(\varphi^n(x_0))$, where $\varphi\colon X\dasharrow X$ and $f\colon X\dasharrow\mathbb{P}^1$ are rational maps
defined over $K$ and $x_0\in X$ is a point whose forward orbit avoids the indeterminacy loci
of $\varphi$ and $f$.  We show that the set of $n$ for which $a_n\in G$ is a finite union of arithmetic progressions along with
a set of upper Banach density zero.  In addition, we show that if $a_n\in G$ for every $n$ and the $\varphi$ orbit
of $x$ is Zariski dense in $X$ then {there is} a multiplicative torus $\mathbb{G}_m^d$ and maps $\Psi:\mathbb{G}_m^d \to \mathbb{G}_m^d$
and $g:\mathbb{G}_m^d \to \mathbb{G}_m$ such that $a_n = g\circ \Psi^n(y)$ for some $y\in \mathbb{G}_m^d$. We then describe various applications of  our results.
\end{abstract}

\maketitle



\section{Introduction}


\subsection{Notation}
\label{subsec:notation}

In this note, we consider \textit{rational dynamical systems}, which are given by a pair $(X,\varphi)$, where $X$ is always an irreducible quasiprojective variety defined over a field $K$, and $\varphi:X\dashrightarrow X$ is a rational map. The forward $\varphi$-orbit of a point $x_0\in X$ is given by
\[ O_\varphi(x_0):=\{x_0,\varphi(x_0),\varphi^2(x_0),\ldots\} \]
 as long as this orbit is defined (\textit{i.e.}, $x_0$ is outside the indeterminacy locus of $\varphi^n$ for every $n\geq 0$).

Throughout our paper, we let $\N:=\{1,2,3,\ldots\}$ and $\N_0:=\N\cup\{0\}$. If $R$ is a ring then $R^*$ is its multiplicative group of units.
An arithmetic progression is a set of the form $\{a+bn\}_{n\geq 0}\subseteq \N_0$ where $a,b\in \N_0$. We take a singleton to be an arithmetic progression with $b=0$. A subset $N\subseteq \N_0$ is called \textit{eventually periodic} if it is a union of finitely many arithmetic progressions. Finally, the \emph{(upper) Banach density} of a subset $S\subseteq \N_0$ is
\[ \delta(S):=\limsup_{|I|\rightarrow\infty}\frac{|S\cap I|}{|I|}, \]
where $I$ ranges over all non-empty intervals of $\N_0$ (see~\cite[Definition 3.7]{Fur81}).


\subsection{Our main results}
\label{subsec:results}

The purpose of this short note is to show that the results of \cite{BCE} hold (after slight necessary modifications to the statements and proofs) over fields of positive characteristic. So, we prove the following result.
\begin{thm}
\label{main theorem}
Let $X$ be a quasiprojective variety over a field $K$ of positive characteristic, let $\varphi:X\dashrightarrow X$ be a rational map, let $f:X\dashrightarrow \mathbb{P}^1$ be a rational function, and let $G\subset K^*$ be a finitely generated subgroup. If $x_0\in X(K)$ is a point with well-defined forward $\varphi$-orbit that also avoids the indeterminacy locus of $f$, then the set
\[{ N :=  \left\{n\in \N_0 : f(\varphi^n(x_0)) \in G\right\}} \]
is a finite union of arithmetic progressions along with a set of upper Banach density zero.
\end{thm}

Our Theorem~\ref{main theorem} extends to fields of positive characteristic the result of \cite[Theorem~1.1]{BCE}. 

The case $N=\N_0$ in Theorem~\ref{main theorem} can be easily achieved: let $T:=\mathbb{G}_m^d$ be a $d$-dimensional multiplicative torus.  Then an endomorphism $\varphi$ of $T$ is a map of the form
$$(x_1,\ldots ,x_d)\mapsto \left(c_1\prod_{j} x_j^{a_{1,j}},\ldots ,c_d\prod_j x_j^{a_{d,j}}\right).$$
 {Now if} we begin with a point  $x_0=(\beta_1,\ldots ,\beta_d)$, then every point in the orbit of $x_0$ has coordinates in the multiplicative group $G$ generated by $$c_1,\ldots ,c_d,\beta_1,\ldots, \beta_d.$$  In particular, if $f:T\to \mathbb{P}^1$ is a map of the form $(x_1,\ldots ,x_d)\mapsto \kappa x_1^{p_1}\cdots x_d^{p_d}$ with $\kappa\in G$ then $f\circ \varphi^n(x_0)\in G$ for every $n\ge 0$. It is interesting to see that, \emph{actually}, each  dynamical system $(X,\varphi)$ with $N=\N_0$ is controlled  by one of this form, in the sense given in the following result.

\begin{thm} \label{main theorem 2}
Let $K$ be a field of positive characteristic, let $X$ be a quasiprojective variety with a dominant self-map {$\varphi:X\dashrightarrow X$, and let $f:  X \dashrightarrow \mathbb{P}^1$} be a dominant rational map, all defined over $K$.  Suppose that {$x_0\in X$} has the following properties:
\begin{enumerate}
\item every point in the orbit of {$x_0$} under $\varphi$ avoids the indeterminacy loci of $\varphi$ and $f$;
\item {$O_{\varphi}(x_0)$} is Zariski dense;
\item there is a finitely generated multiplicative subgroup $G$ of $K^*$ such that $f\circ \varphi^n({x_0})\in G$ for every $n\in \mathbb{N}_0$.
\end{enumerate}
Then {there exists} a rational map $\Theta:X\dashrightarrow \mathbb{G}_m^d$ for some nonnegative integer $d$, and a dominant endomorphism $\Phi:\mathbb{G}_m^d\rightarrow \mathbb{G}_m^d$ such that the following diagram commutes
\begin{center}
    \begin{tikzpicture}[node distance=1.7cm, auto]
        \node (X1) {$X$};
        \node (X2) [right of=X1] {$X$};
        \node (Ad1) [below of=X1] {$\mathbb{G}_m^d$};
        \node (Ad2) [below of=X2] {$\mathbb{G}_m^d.$};

        \draw[->,dashed] (X1) to node {$\varphi$} (X2);
        \draw[->,dashed] (X1) to node [swap] {$\Theta$} (Ad1);
        \draw[->,dashed] (X2) to node {$\Theta$} (Ad2);
        \draw[->] (Ad1) to node [swap] {$\Phi$} (Ad2);
    \end{tikzpicture}
\end{center}
Moreover, $O_{\varphi}(x_0)$ avoids the indeterminacy locus of $\Theta$ and $f=g\circ \Theta$, where $g:\mathbb{G}_m^d \to \mathbb{G}_m$ is a map of the form
\[ g(t_1,\ldots,t_d) = Ct_1^{i_1}\cdots t_d^{i_d}\]
for some $i_1,\ldots,i_d\in \Z$ and some $C\in G$.
\end{thm}

Theorem~\ref{main theorem 2} is a positive characteristic variant of \cite[Theorem~1.2]{BCE}. However, there is an important difference between the conclusion of \cite[Theorem~1.2]{BCE} and the conclusion of Theorem~\ref{main theorem 2}: one cannot expect that the map $\Theta$ is dominant, as shown in \cite[Example~3.7]{BCE}.

One can interpret Theorem~\ref{main theorem 2} as saying that if the entire orbit of a point under a self-map has some ``coordinate'' that lies in a finitely generated multiplicative group, then there must be some geometric reason which is causing this phenomenon: in this case, it is that the dynamical behaviour of the orbit is completely
determined by the behaviour of a related system associated with a multiplicative torus.  In fact, similar to the results from \cite{BCE}, one can prove a more general version of this result involving semigroups of maps (see Theorem~\ref{cor:semi} for the precise formulation). Furthermore, as a consequence of Theorem~\ref{main theorem 2}, we get the following characterization of orbits whose values lie in a finitely generated subgroup of $K^*$, which shows that on arithmetic progressions they are well-behaved.
\begin{cor} 
\label{orbit}
Let $K$ be a field of positive characteristic and let $X$ be a quasiprojective variety with a dominant self-map {$\varphi:X \dashrightarrow X$ and let $f:\mathbb{X}\dashrightarrow \mathbb{P}^1$} be a dominant rational map, all defined over $K$.  Suppose that ${x_0\in X}$ has the following properties:
\begin{enumerate}
\item every point in the orbit of {$x_0$} under $\varphi$ avoids the indeterminacy loci of $\varphi$ and $f$;
\item there is a finitely generated multiplicative subgroup $G$ of $K^*$ such that $f\circ \varphi^n({x_0})\in G$ for every $n\in \mathbb{N}_0$.
\end{enumerate}
Then there are integers $\ell$ and $L$ with $\ell\ge 0$ and $L>0$ such that if $h_1,\ldots ,h_m$ generate $G$ then there are integer valued linear recurrences $b_{j,1}(n),\ldots ,b_{j,m}(n)$ for $j\in \{0,\ldots ,L-1\}$ such that
\[ {f\circ \varphi^{Ln+j} (x_0) = \prod_{i=1}^m h_i^{b_{j,i}(n)}}\]
for each $n\ge \ell$.
\end{cor}

As explained in Subsection~\ref{subsec:Polya}, one of the motivations for considering when a dynamical sequence $\{f(\Phi^n(\alpha))\}_{n\in\N_0}$ (for some self-map $\Phi$ on a quasiprojective variety $X$ endowed with a rational map $f:X\dashrightarrow \mathbb{P}^1$)  takes values in a finitely generated multiplicative group comes from studying $D$-finite power series. Already, this 
 question along with other related questions have  been considered in the case of self-maps of $\mathbb{P}^1$ in \cite{BOSS, BOSS2, BNZ}. Questions about $D$-finite series are of less relevance in the positive characteristic setting, but we are nevertheless able to apply our results to rational series over positive characteristic fields whose coefficients take values in a finitely generated group.  In particular, we prove a positive characteristic analogue of a classical theorem of P\'olya \cite{Pol}.
 \begin{thm}
\label{thm:D-finite_2}
Let $K$ be a field of positive characteristic, let $G\subset K^*$ be a finitely generated subgroup, and let 
$$F(x)=\sum_{n\ge 0}a_nx^n$$
be a rational power series defined over $K$, and suppose that $a_n\in G\cup\{0\}$ for all $n\ge 0$.  Then there is some $M\ge 1$ and some $N\ge 0$ such that 
$$F(x) = P(x) + \sum_{b=0}^{M-1} \mu_b x^{b+MN}/(1-\delta_b x^M),$$ where $P(x)$ is a polynomial of degree at most $MN-1$ with coefficients in $G\cup \{0\}$, $\delta_0,\ldots ,\delta_{M-1}\in G$ and $\mu_0,\ldots ,\mu_{M-1}\in G\cup \{0\}$.  
\end{thm}
This theorem was proved for integer-valued linear recurrences by P\'olya \cite{Pol} and later extended to characteristic fields (and $D$-finite series) by B\'ezivin \cite{Bez}.  The techniques applied were not, however, amenable dealing with this question in the positive characteristic setting.  We show that by applying work on $S$-unit equations due to Derksen and Masser \cite{DM12}, one can circumvent the obstacles that arise in positive characteristic.  


\subsection{The dynamical Mordell-Lang conjecture in positive characteristic}

Another motivation for our question comes from the Dynamical Mordell-Lang Conjecture (stated below as Conjecture~\ref{conj:DML_p}).
 
In the case when $G$ is the trivial group, Theorem~\ref{main theorem} was already known (see \cite{BGT}, for example), and was motivated  by a conjecture of Denis~\cite{Denis}. Furthermore, the special case when $G$ is the trivial group is a weakening of the Dynamical Mordell-Lang Conjecture in characteristic $p$ (see \cite{G-TAMS} and also \cite{DML-book} for a comprehensive discussion of the Dynamical Mordell-Lang Conjecture).
\begin{conj}
\label{conj:DML_p}
Let $X$ be a quasiprojective variety defined over a field $K$ of characteristic $p$, endowed with an endomorphism $\Phi$. Then for any point $\alpha\in X(K)$ and any subvariety $Y\subseteq X$, the set 
$$S:=\left\{n\in\N_0\colon \Phi^n(\alpha)\in Y(K)\right\}$$
is a finite union of arithmetic progressions along with finitely many sets of the form 
\begin{equation}
\label{eq:form_set}
\left\{\sum_{i=1}^r d_ip^{k_in_i}\colon n_i\in\N_0\right\},
\end{equation}
for some given $r\in\N$ and some given rational numbers $d_i$ and nonnegative integers $k_i$ (for $i=1,\dots, r$).
\end{conj}
So, Conjecture~\ref{conj:DML_p} predicts that besides finitely many arithmetic progressions, the {return set} corresponding to the intersection of an orbit with a subvariety may also contain a set of Banach density $0$ of a very special type~\eqref{eq:form_set}. 

It is natural to ask whether the set of Banach density $0$ appearing in the {return set} $N$ from the conclusion of Theorem~\ref{main theorem} comprises finitely many sets of the same special form~\eqref{eq:form_set}. However, the Dynamical Mordell-Lang Conjecture in positive characteristic is a very difficult question, even in the special case of endomorphisms of the multiplicative group $\mathbb{G}_m^N$, in which case it reduces to some deep Diophantine questions (for more details, see \cite{CGSZ}). More precisely, using the same construction as in \cite[Section~4]{CGSZ} (especially, see \cite[Proposition~4.3]{CGSZ}), one can construct some endomorphism $\Phi$ of $\mathbb{G}_m^N$ (defined over $\F_p(t)$) along with a starting point $\alpha\in\mathbb{G}_m^N(\F_p(t))$, a suitable rational function $f:\mathbb{G}_m^N\dashrightarrow \mathbb{A}^1$, and a finitely generated subgroup $G\subset \mathbb{G}_m(\F_p(t))$ such that for some given $r\in\N$, the set 
\begin{equation}
\label{eq:return_set}
N:=\left\{n\in\N_0\colon f(\Phi^n(\alpha))\in G\right\}
\end{equation}
is precisely the set of all $n\in\N_0$ for which there exist $m_1,\dots,m_r\in\N_0$ such that 
\begin{equation}
\label{eq:hard}
n^2=p^{m_1}+p^{m_2}+\cdots + p^{m_r}.
\end{equation}
The Diophantine equation~\eqref{eq:hard} is \emph{very difficult} (already when $r\ge 5$) and all one can show with the current Diophantine methods is that  the set of all $n\in\N_0$ which satisfy an equation of the form~\eqref{eq:hard} has natural density $0$ (see \cite{GOSS}). Obtaining the precise description of the return set $N$ from \eqref{eq:return_set} as a finite union of sets of the form~\eqref{eq:form_set} is beyond the known results available in the literature. 

The fact that the problem in positive characteristic turns out to be more subtle than the corresponding question in characteristic $0$ is encountered in many similar questions in arithmetic geometry (such as the classical Mordell-Lang conjecture, see \cite{Hrushovski, Rahim-Tom} for the corresponding results in  characteristic $p$) or arithmetic dynamics (such as the Zariski dense orbit conjecture in positive characteristic, see \cite{G-Sina}).  So, we find it interesting that the results of \cite{BCE}, which involve a language akin to both the Dynamical Mordell-Lang Conjecture and to the classical Mordell-Lang conjecture, hold with suitable modifications over fields of positive characteristic.


\subsection{Plan of our paper}
\label{subsec:plan}

The proofs of our results follow closely the strategy from \cite{BCE}; the only significant difference between the two arguments appears in the proof of Lemma~\ref{dichotomy}. In Section~\ref{sec:proof} we explain the differences from the arguments of~\cite{BCE}, including the proof of Lemma~\ref{dichotomy}, in order to derive the proof for our results in positive characteristic. We conclude with presenting several applications (including Theorem~\ref{thm:D-finite_2}) of our results in Section~\ref{sec:applications}.


\section{Proof of our results}
\label{sec:proof}

The strategy of proof follows the arguments employed in \cite{BCE}. We proceed by describing the differences appearing in each section of the proof from \cite{BCE} when working in positive characteristic.


\subsection{Linear sequences in abelian groups}
\label{subsec:linear}

The contents of \cite[Section~2]{BCE} provide a general background for sequences in abelian groups and the corresponding results also hold with identical proofs in the case when the relevant rings have positive characteristic. 

For example, the translation in positive characteristic of   \cite[Proposition~2.9]{BCE} asserts the following result: given a field $K$ which is a finitely generated extension of $\F_p$ and given a sequence $\{u_n\}_{n\in\N_0}\subset K^*$ satisfying a multiplicative $\N_0$-quasilinear recurrence (see \cite[Definition~2.1]{BCE}), then $\{u_n\}_{n\in\N_0}$ satisfies a multiplicative linear recurrence and moreover,  if $H$ is a finitely generated subgroup of $K^*$, then $\{n\in\N_0\colon u_n\in H\}$ is eventually periodic. Indeed, working with a field $K$ which is finitely generated over $\F_p$, one views $K$ as a finite extension of some rational function field $\F_p(t_1,\dots, t_m)$. So, considering $R:=\F_p[G,t_1,\dots, t_m]$, the finitely generated ring spanned by the $t_i$'s and the elements of the finitely generated subgroup $G\subset K^*$, its integral closure $\overline{R}$ inside $K$ is also a finitely generated $\F_p$-algebra, according to \cite[Corollary~13.13]{Eisenbud}. Then the group of units of $\overline{R}^*$ is again finitely generated (by \cite{Roquette}) and the rest of the proof of \cite[Proposition~2.9]{BCE} follows identically for such fields $K$ of characteristic $p$.


\subsection{Multiplicative dependence of certain rational functions}

The goal of \cite[Section~3]{BCE} is to convert the statement of \cite[Theorem~1.2]{BCE} into a problem about linear recurrence sequences as developed in \cite[Section~2]{BCE}. Since the results of \cite[Section~2]{BCE} hold in positive characteristic (as explained in Section~\ref{subsec:linear}), then our goal is to see how the proofs of \cite[Section~3]{BCE} can be changed so that the main results would hold for fields of characteristic $p$. We prove below the characteristic $p$ variant of \cite[Lemma~3.2]{BCE}, which requires a slightly different argument in order to eliminate the use of the $S$-unit equation in characteristic $0$.

\begin{lemma}
\label{dichotomy}
Let $K$ be an algebraically closed field of characteristic $p$ with transcendence degree $e<\infty$ over $\mathbb{F}_p$, let $G$ be a finitely generated multiplicative subgroup of $K^{*}$ of rank $r$, let $X$ be an irreducible quasiprojective variety over $K$ of dimension $d$, and let $g_0,\ldots,g_{N-1}\in K(X)$ be $L$ rational functions on $X$ with $N=(d+e+1)^r$. If 
\begin{equation}
\label{eq:def_X_g}
X_G:= X_G(g_0,\ldots,g_{N-1}):=\left\{x\in X(K)\colon g_0(x),g_1(x),\dots,g_{N-1}(x)\in G \right\} 
\end{equation}
is Zariski dense in $X$, then $g_0,\ldots,g_{N-1}$ are multiplicatively dependent.
\end{lemma}

\begin{proof} We prove this by induction on $r$.  If $r=0$ then $G$ is a finitely generated torsion subgroup of $K(X)$ and hence is finite.  It follows that if $g_0\in K(X)$ has the property that $X_G(g_0)$ is Zariski dense then $g_0$ must be constant since $X$ is irreducible; moreover, this constant must be in the torsion group $G$ and so $g_0^m=1$ for some $m$, giving us the result in this case.  

Now we assume that the results holds whenever $r<s$ with $s\ge 1$ and we consider the case when $r=s$.
Let $L=(d+e+1)^{s-1}$.  

We fix a rank one discrete valuation $\nu$ of $K(X)$ such that $\nu$ is not identically zero on $G$ (note that by our assumption, not all elements of $G$ are contained in $\overline{\F_p}$ because $G$ has positive rank).  Then after renormalizing, we may assume that $\nu|_G : G\to \Z$ is surjective and we let $G_0$ denote the kernel, which is a finitely generated subgroup of $G$ of rank $s-1$.    

Since the field extension $K(X)/\mathbb{F}_p$ has transcendence degree $d+e$, then for each $j=0,\ldots ,L-1$, the functions $f_{j,0}:=g_{j(d+e+1)},\ldots,f_{j,d+e}:=g_{j(d+e+1)+d+e}$ are algebraically dependent over $\mathbb{F}_p$. Thus there is a nontrivial polynomial relation
\[ \sum_{i_0,\ldots,i_{d+e}} c_{i_0\cdots i_{d+e},j}f_{j,0}^{i_0}\cdots f_{j,d+e}^{i_{d+e}}=0 \]
where $c_{i_0\cdots i_{d+e},j}\in \mathbb{F}_p$ and the sum is over a finite set of indices in $\N_0^{d+e+1}$; this holds on some open subset of $X$. For each $j\in\{0,\ldots ,L-1\}$, we let $I_j$ be the (finite) set of indices $\alpha=(i_0,\ldots,i_{d+e})\in \N_0^{d+e+1}$ where $c_{i_0\cdots i_{d+e},j}$ is nonzero. For $\gamma=(i_0,\ldots,i_{d+e})\in \Z^{d+e+1}$, we set
\[{f_{\gamma,j}:=f_{j,0}^{i_0}\cdots f_{j,{d+e}}^{i_{d+e}} \quad \text{and} \quad c_{\gamma,j} := c_{i_0\cdots i_{d+e},j}. }\]
Then for every $y\in X_G$ and each $j\in \{0,\ldots ,L-1\}$, there are two distinct elements $\alpha_j, \beta_j$ in the $|I_j|$-tuple $(c_{\alpha,j} f_{\alpha,j}(y))_{\alpha\in I_j}$ having the same valuation under $\nu$.
Given $\alpha_0,\ldots ,\alpha_{L-1}$, $\beta_0,\ldots ,\beta_{L-1}$ with each $\alpha_j,\beta_j\in I_j$ distinct, we let $X_{G,(\alpha_0,\ldots, \alpha_{L-1};\beta_0,\ldots, \beta_{L-1})}$ denote the set of points $y\in X_G$
such that $\nu(f_{\alpha_j,j}(y)/f_{\beta_j,j}(y))=0$. Then there is some element 
$$((\alpha_0,\beta_0),\cdots , (\alpha_{L-1},\beta_{L-1}))\in \prod_{j=0}^{L-1} I_j^2$$ 
with $\alpha_j\neq\beta_j$ for each $j$ such that $Z:=X_{G,(\alpha_0,\ldots, \alpha_{L-1};\beta_0,\ldots, \beta_{L-1})}$ is Zariski dense in $X$.  
Now let $t_j: =f_{\alpha_j,j}/f_{\beta_j,j}$ for $j=0,\ldots ,L-1$.  Then by construction, 
$X_{G_0}(t_0,\ldots ,t_{L-1}) \supseteq Z$ and thus is Zariski dense.  Then since $G_0$ has rank $s-1$, the induction hypothesis gives that $t_0,\ldots ,t_{L-1}$ are multiplicatively dependent.  A multiplicative dependence between these functions then gives a dependence between $g_0,\ldots , g_{L(d+e+1)}$, and so the result follows by induction.
\end{proof}


\subsection{Conclusion of our proof}

Then the proof of Theorem~\ref{main theorem 2} follows identically as the proof of \cite[Theorem~1.2]{BCE} follows from \cite[Section~3]{BCE}; moreover, one obtains the following more general statement than Theorem~\ref{main theorem 2} working with finitely many self-maps of $X$. 

\begin{thm}
\label{cor:semi}
Let $K$ be an algebraically closed field of positive characteristic, let $G$ be a finitely generated subgroup of $K^*$, let $X$ be a quasiprojective variety over $K$, let $\varphi_1,\ldots ,\varphi_m$ be dominant rational self-maps of $X$ and we let $S$ denote the monoid generated by these maps under composition. We let $f:X\dashrightarrow \mathbb{P}^1$ be a non-constant rational map and assume that $x_0\in X(K)$ has the property that its forward orbit under $S$ is Zariski dense and each point avoids the indeterminacy loci of the maps $\varphi_1,\ldots ,\varphi_m$ and $f$.

If {$u_{\varphi} = f(\varphi(x_0))\in G$} for every $\varphi$ in the monoid $S$   then {there exists} a rational map $\Theta:X\dashrightarrow \mathbb{G}_m^d$ {with $d\leq {\rm dim}(X)$} that is defined at each point in {$O_{\varphi}(x_0) = (\varphi(x_0))_{\varphi \in S}$} and endomorphisms $\Phi_1,\ldots ,\Phi_m:\mathbb{G}_m^d\rightarrow \mathbb{G}_m^d$ such that the following diagram commutes
\begin{center}
    \begin{tikzpicture}[node distance=2.7cm, auto]
        \node (X1) at (0,0) {$X$};
        \node (X2) at (3,0) {$X$};
        \node (Ad1) at (0,-2) {$\mathbb{G}_m^d$};
        \node (Ad2) at (3,-2) {$\mathbb{G}_m^d.$};

        \draw[->,dashed] (X1) to node  {$\varphi_1,\ldots ,\varphi_m$} (X2);
        \draw[->,dashed] (X1) to node [swap]{$\Theta$} (Ad1);
        \draw[->,dashed] (X2) to node {$\Theta$} (Ad2);
        \draw[->] (Ad1) to node [swap] {$\Phi_1,\ldots ,\Phi_m$} (Ad2);
    \end{tikzpicture}
\end{center}
\end{thm}

Theorem~\ref{main theorem 2} is a special case of Theorem~\ref{cor:semi} when the semigroup $S$ from Theorem~\ref{cor:semi} is cyclic.

The main difference between Theorem~\ref{cor:semi} as opposed to its counterpart in  characteristic $0$ (which is \cite[Corollary~3.5]{BCE}) consists in the fact that the map $\Theta$ appearing in the conclusion from Theorem~\ref{cor:semi} is no longer dominant. The reason for this is that in the final part of the proof of \cite[Corollary~3.5]{BCE}, one employs Laurent's classical result \cite{Laurent} for the Mordell-Lang conjecture for the multiplicative group to infer that an irreducible subvariety of an algebraic torus must be itself a translate of a subtorus if it contains a Zariski dense set of points from a given finitely generated subgroup of the torus; this conclusion does not hold in positive characteristic. All one obtains, with the notation as in Theorem~\ref{cor:semi} is that the image of $\Theta$ is a translate of a subvariety $Y\subseteq \mathbb{G}_m^d$ defined over a finite field (as per Hrushovski's proof \cite{Hrushovski} of the Mordell-Lang conjecture in positive characteristic; see also \cite{Rahim-Tom}). 

Finally, similar to the proof of \cite[Corollary~1.3]{BCE}, one derives the conclusion of Corollary~\ref{orbit}; then the proof of Theorem~\ref{main theorem} follows verbatim as the proof of \cite[Theorem~1.1]{BCE} from \cite[Section~4]{BCE}.


\section{Applications of our results}
\label{sec:applications}

The applications from \cite[Sections~5~and~6]{BCE} hold with almost identical proofs also in the case of fields of positive characteristic. We state below applications in a couple different directions.


\subsection{Heights}
\label{subsec:heights}

In \cite{BGS} and \cite{BFS}, a broad dynamical framework giving rise to many
classical sequences from number theory and algebraic combinatorics is developed. This is accomplished by considering \emph{dynamical sequences}, which are sequences of the 
form $ f\circ \varphi^n(x_0)$, where $(X,\varphi)$ is a rational
dynamical system, $f:X\dashrightarrow \mathbb{P}^1$ is a rational map, and $x_0\in X$. Similar to \cite{BGS, BFS}, Corollary~\ref{orbit} provides an interesting ``gap'' about heights of points in the forward orbit of a self-map $\varphi$ for varieties and maps defined over a function field over a finite field. 

So, for a finitely generated field $K$ over its prime field $\F_p$, we view $K$ as the function field of a projective geometrically irreducible variety $\mathcal{V}$ defined over a finite field $\F_q$; at the expense of replacing $K$ by a finite extension, we may assume $\mathcal{V}$ is smooth \cite[Remark 4.2]{DeJ96}. We let $\Omega_{\mathcal{V}}$ be the set of inequivalent absolute values corresponding to the irreducible divisors of $\mathcal{V}$. Then one can construct the Weil height $h(\cdot ):=h_{\mathcal{V}}(\cdot )$ for the points in $K$ corresponding to the places in $\Omega_{\mathcal{V}}$; for further background on height functions, we refer the reader to \cite[Chapter~2]{BG06} and \cite[Chapter~3]{GTM241}. 

The following result is a positive characteristic analogue of \cite[Theorem~5.1]{BCE}.
\begin{thm}
Let $K$ be a finitely generated extension of $\F_p$ and let $X$ be a  quasiprojective variety defined over $K$, endowed with a dominant self-map $\varphi:X\dashrightarrow X$ along with a rational map    $f:{X}\dashrightarrow \mathbb{P}^1$, both defined also over $K$.  Suppose that $x_0\in X(K)$ has the following properties:
\begin{enumerate}
\item every point in the orbit of $x_0$ under $\varphi$ avoids the indeterminacy loci of $\varphi$ and $f$;
\item there is a finitely generated multiplicative subgroup $G$ of $K^*$ such that $f\circ \varphi^n(x_0)\in G$ for every $n\in \mathbb{N}_0$.
\end{enumerate}
If $h(f\circ \varphi^n(x_0)) = {\rm o}(n^2)$ then the sequence
$(f\circ \varphi^n(x_0))_n$ satisfies a linear recurrence.  More precisely, there exists an integer $L\ge 1$ such that for each $j\in \{0,\ldots, L-1\}$ there are $\alpha_j,\beta_j\in G$ such that for all $n$ sufficiently large we have
$$f\circ \varphi^{Ln+j}(x_0) = \alpha_j \beta_j^n.$$
\label{orbit2}
\end{thm}

\begin{proof}
The arguments are almost identical with the ones employed in the proof of \cite[Theorem~5.1]{BCE}. For example, the proof of \cite[Theorem~5.1]{BCE} uses a classical result of Schlickewei \cite[Theorem 1.1]{Schlick} which gives that for a finitely generated subgroup $H\subset \overline{\Q}^*$, there exist $g_1,\ldots ,g_d\in H$ such that {every element} of $H$ can be expressed uniquely in the form
$\zeta g_1^{n_1}\cdots g_d^{n_d}$ with $\zeta$ {being a root} of unity and $n_1,\ldots ,n_d\in \mathbb{Z}$ and such that we have the following lower bound for the Weil height of the points in $H$:
\begin{equation}
\label{eq:Schlick}
h(\zeta g_1^{n_1}\cdots g_d^{n_d})\ge \max_{1\le i\le d} \{|n_i| 4^{-d} h(g_i)\}.
\end{equation}
An analogue of this result can be easily derived for function fields $K$ as in our setting. Indeed, viewing (as before) $K$ as the function field of a smooth projective geometrically irreducible variety $\mathcal{V}$ over some finite field $\F_q$, we let $\Omega_{\mathcal{V}}$ the set of inequivalent absolute values associated with the irreducible divisors of $\mathcal{V}$. Then letting $r$ be the rank of the  torsion-free subgroup $G_0$ of $G$, we can find suitable generators $g_1,\dots, g_r$ of $G_0$ along with places $v_1,\dots, v_r$ of $K$ (corresponding to irreducible divisors of $\mathcal{V}$), so that the following properties are satisfied by the absolute values $|\cdot|_{v_i}\in\Omega_{\mathcal{V}}$:
\begin{itemize}
\item[(i)] $\left|g_i\right|_{v_i}>1$ for $i=1,\dots, r$; and 
\item[(ii)] $\left|g_j\right|_{v_i}=0$ for each $i\ne j$. 
\end{itemize}
(We point out that such places can be produced inductively, using row reduction to obtain generators for which (i) and (ii) hold.)
Conditions~(i)-(ii) guarantee that for any integers $n_1,\dots, n_r$, we have that the Weil height $h:K\longrightarrow \mathbb{R}_{\ge 0}$ satisfies
$$h\left(\prod_{i=1}^r g_i^{n_i}\right)>c\cdot \max_{i=1}^r |n_i|,$$
for some positive constant $c$ depending only on the places $v_1,\dots ,v_r$ and the points $g_1,\dots, g_r$. This allows us to apply the same arguments as in the proof of \cite[Theorem~5.1]{BCE} to conclude our proof of Theorem~\ref{orbit2}.
\end{proof}


\subsection{A theorem of P\'olya in positive characteristic}
\label{subsec:Polya}

The class of dynamical sequences includes all sequences whose generating functions are $D$-finite, i.e., those satisfying homogeneous linear differential
equations with rational function coefficients (see also \cite[Definition~6.1]{BCE}).  This is an important class of power series since it
appears ubiquitously in algebra, combinatorics, and number theory; we refer the reader to \cite[Sections~1~and~6]{BCE} for a comprehensive discussion of the $D$-finite power series and their applications.  In positive characteristic, however, the class of $D$-finite power series is considerably less interesting due to the fact that the derivations are nilpotent.  

We notice that the study of power series with coefficients in a finitely generated subgroup $G$ of the multiplicative group of
a field enjoys a long history, going back at least to the early 1920s, with the pioneering work of P\'olya \cite{Pol},
who characterized rational functions whose Taylor expansions at the origin have coefficients lying in a finitely generated multiplicative subgroup of $\Z$.
P\'olya's results were later extended to $D$-finite power series by B\'ezivin \cite{Bez}, although B\'ezivin's work was necessarily done in characteristic zero for the reasons stated in the preceding paragraph.  

Many of the results of \cite[Section~6]{BCE} go through with minimal changes when working over a field of positive characteristic, but in the case of $D$-finite series what one cannot guarantee is that such a sequence arises as a dynamical system.  In cases where this can be done, there are no additional obstacles when working in positive characteristic.  In particular, if one considers rational power series $F(x)=\sum a_n x^n$, then one has $a_n=w^t A^n v$, where $A$ is an invertible matrix and $w,v$ are column vectors, for all $n$ sufficiently large; and so a tail of the sequence $a_n$ can be realized as a dynamical sequence, even in positive characteristic.  In particular, the following result follows verbatim from \cite[Theorem~1.4]{BCE}, as a consequence of both \cite{BGT} and also of our Theorem~\ref{main theorem}.
\begin{thm}
\label{thm:D-finite}
Let $K$ be a field of characteristic $p$, let $G\subset K^*$ be a finitely generated subgroup, and let 
$$F(x)=\sum_{n\ge 0}a_nx^n$$
be a rational power series defined over $K$. Consider the sets
$$N:=\left\{n\in\N_0\colon a_n\in G\right\}\text{ and }N_0:=\left\{n\in\N_0\colon a_n\in G\cup\{0\}\right\}.$$
Then both $N$ and $N_0$ are expressible as a union of finitely many arithmetic progressions along with a set of Banach density $0$. 
\end{thm}

As mentioned before, an analogue of P\'olya's \cite{Pol} has apparently not been worked out in positive characteristic.  This is perhaps not surprising since the theory of linear recurrences in positive characteristic has historically lagged behind the development of the theory in characteristic zero; for example, Derksen's \cite{Der} analogue of the Skolem-Mahler-Lech theorem in positive characteristic was only proved in 2007. We thus give the positive characteristic analogue of P\'olya's theorem to fill a gap in the literature.  
\begin{proof}[Proof of Theorem \ref{thm:D-finite_2}]
It is well-known (see, for example, \cite[Proposition 2.5.1.4]{DML-book}) that there exist $s\ge 0$ and $m\ge 0$ and
$\alpha_1,\alpha_s\in \bar{K}^*$ and constants $c_{i,j}\in \bar{K}$ such that 
$$a_n = \sum_{i=1}^s \sum_{j=0}^m c_{i,j} n^j \alpha_i^n$$ for $n$ sufficiently large.  
Then since we are working in characteristic $p$, we have $a_{pn+k} = \sum_{i=1}^s b_{i,k}' \alpha_i^n$ for some constants $b_{i,k}'\in \bar{K}$, for all sufficiently large $n$.  In particular, there are $M,N\ge 1$ such that for each $b\in \{0,1,\ldots ,M-1\}$ we have 
for all $n\ge N$, $a_{Mn+b}$ is expressible in the form $$\sum_{i=1}^t e_i \beta_i^n$$ for some nonzero $\beta_1,\ldots ,\beta_t\in \bar{K}^*$ with $\beta_i/\beta_j$ not a root of unity for $i\neq j$ and some constants $e_1,\ldots ,e_t\in \bar{K}^*$.  Now let us consider a sequence of the form
$\sum_{i=1}^t e_i \beta_i^n$ which satisfies the above conditions and which takes values in $G\cup \{0\}$;
i.e.,
$$\sum_{i=1}^t e_i \beta_i^n \in G\cup \{0\}$$ for all $n$.  We claim that $t\le 1$. To see this, let $H$ denote the multiplicative subgroup of $\bar{K}^*$ generated by $G$ and $e_1,\ldots ,e_t, \beta_1,\ldots ,\beta_t$.
Then if $t\ge 2$ then by a result of Derksen and Masser (see \cite[Proposition 2.2]{BN18} for the precise statement used here)
there is some finite set $\mathcal{S}\subseteq \bar{K}^*$ such  
that for every $n$, we have $$(e_1/e_2) \cdot (\beta_1/\beta_2)^n = \alpha_{0,n}\alpha_{1,n}^{p^{i_{1,n}}}\cdots \alpha_{t-1,n}^{p^{i_{t-1,n}}}$$ for some 
$\alpha_{0,n},\ldots ,\alpha_{t-1,n}\in \mathcal{S}$ and $i_{1,n},\ldots ,i_{t-1,n}\in \{0,1,\ldots \}$.  In particular, since $\mathcal{S}$ is finite, there exist $\gamma_0,\ldots ,\gamma_{t-1}\in \mathcal{S}$ such that for $n$ in a subset $T$ of the natural numbers of positive density, we have
$$(e_1/e_2) \cdot (\beta_1/\beta_2)^n = \gamma_{0}\gamma_{1}^{p^{i_{1,n}}}\cdots \gamma_{t-1}^{p^{i_{t-1,n}}}.$$
Let $L$ be the field extension of $\bar{F}_p$ generated by $G_0$.
Now since $\beta_1/\beta_2$ is not a root of unity, there is some rank-one discrete valuation $\nu$ of $L$ such that $\nu(\beta_1/\beta_2)=a>0$.  Letting 
$c_i = \nu(\gamma_i)$ for $i=0,\ldots ,t-1$ and letting $a'=\nu(e_1/e_2)$ we see that 
\begin{equation}
\label{eq:aa}
a'+na = c_0 + c_1 p^{i_{1,n}}+\cdots + c_{t-1} p^{i_{t-1,n}}
\end{equation} for all $n$ in the positive density set $T$.  But arguing as in \cite[\S 2.3, Case 1]{GOSS}, taking $Q(n)=a'+na$ and the $\lambda_j=p$ for all $j$, we see that the set of $n$ for which Equation (\ref{eq:aa}) holds has density zero, a contradiction.  It follows that $t=1$ as desired.  It follows that for each $b\in \{0,\ldots ,M-1\}$,  
$a_{Mn+b}= \mu_b \delta_b^n$ for some $\mu_b,\delta_b\in K$ for $n\ge N$.  Moreover, since $a_{Mn+b}\in G\cup \{0\}$ for all $n$ we see that $\mu_b,\delta_b\in G$ if $\mu_b\delta_b\neq 0$.  It follows that 
$$F(x) = P(x) + \sum_{i=0}^{M-1} \mu_b x^{b+MN}/(1-\delta_b x^M)$$ for some polynomial $P(x)$ of degree $< MN$ with coefficients in $G\cup \{0\}$.  Notice that if $\delta_b$ or $\mu_b$ is zero, we can take $\delta_b=1$ and $\mu_b=0$ instead, so the result follows. 
\end{proof}



\end{document}